\def\draftdate{October 5, 2010}

\documentclass{amsart}
\usepackage{amssymb}
\usepackage{xy}

\xyoption{arrow}
\xyoption{curve}
\xyoption{matrix}
\xyoption{cmtip}
\SelectTips{cm}{}

\newcommand{\fps}[1]{\mathbf{#1}}
\newcommand{\fpm}{\fps{m}}
\newcommand{\fpn}{\fps{n}}
\newcommand{\fpp}{\fps{p}}
\newcommand{\fpz}{\fps{0}}
\newcommand{\fpo}{\fps{1}}
\newcommand{\fpt}{\fps{2}}
\newcommand{\fus}[1]{\underline{#1}}
\newcommand{\fum}{\fus{m}}
\newcommand{\fun}{\fus{n}}
\newcommand{\fup}{\fus{p}}
\newcommand{\fuo}{\fus{1}}
\newcommand{\fuz}{\fus{0}}
\newcommand{\aom}{\vec m}
\newcommand{\aon}{\vec n}
\newcommand{\aop}{\vec p}
\newcommand{\aob}[1]{\vec{#1}}

\newcommand{\Seg}{\prime}
\newcommand{\aPlax}{\aP^{\textrm{lax}}}

\newcommand{\Cat}{\mathbf{Cat}}

\newcommand{\boxbin}{\mathbin{\Box}}
\newcommand{\groth}{\mathbin{\myop\int}}

\DeclareMathOperator{\simp}{\aS}

\mathchardef\varDelta="7101

\def\myop#1{\mathop{\textstyle #1}\limits}
\let\iso\cong
\let\sma\wedge

\renewcommand{\to}{\mathchoice{\longrightarrow}{\rightarrow}{\rightarrow}{\rightarrow}}
\newcommand{\from}{\mathchoice{\longleftarrow}{\leftarrow}{\leftarrow}{\leftarrow}}

\let\catsymbfont\mathcal
\newcommand{\aA}{{\catsymbfont{A}}}

\newcommand{\aC}{{\catsymbfont{C}}}
\newcommand{\aD}{{\catsymbfont{D}}}
\newcommand{\aE}{{\catsymbfont{E}}}
\newcommand{\aF}{{\catsymbfont{F}}}
\newcommand{\aK}{{\catsymbfont{K}}}
\newcommand{\aP}{{\catsymbfont{P}}}
\newcommand{\aN}{{\catsymbfont{N}}}
\newcommand{\aS}{{\catsymbfont{S}}}
\newcommand{\aW}{{\catsymbfont{W}}}
\newcommand{\aX}{{\catsymbfont{X}}}
\newcommand{\aY}{{\catsymbfont{Y}}}
\newcommand{\aZ}{{\catsymbfont{Z}}}

\newcommand{\bZ}{{\mathbb{Z}}}

\def\quickop#1{\expandafter\DeclareMathOperator\csname
#1\endcsname{#1}}
\quickop{id}\quickop{Id}\quickop{op}
\quickop{Hom}
\quickop{Ob}\quickop{diag}
\quickop{hocolim}
\quickop{Sd}\quickop{Ex}

\numberwithin{equation}{section}

\newtheorem{thm}[equation]{Theorem}
\newtheorem*{thm*}{Theorem}
\newtheorem{cor}[equation]{Corollary}
\newtheorem{lem}[equation]{Lemma}
\newtheorem{prop}[equation]{Proposition}
\theoremstyle{definition}
\newtheorem{defn}[equation]{Definition}
\newtheorem{cons}[equation]{Construction}
\newtheorem{notn}[equation]{Notation}
\theoremstyle{remark}

\newtheorem{var}[equation]{Variant}

\def\enref#1{(\ref{#1})}

\newcommand{\term}[1]{\textit{#1}}

\newdir{ >}{{}*!/-5pt/\dir{>}}

\begin{document}

\title%
{An Inverse $K$-Theory Functor}

\author{Michael A. Mandell}
\address{Department of Mathematics, Indiana University,
Bloomington, IN \ 47405}
\email{mmandell@indiana.edu}
\thanks{The author was supported in part by NSF grant DMS-0804272}

\date{\draftdate}

\begin{abstract}
Thomason showed that the $K$-theory of symmetric monoidal categories models
all connective spectra.  This paper describes a new construction
of a permutative category from a $\Gamma$-space, which is then used to re-prove
Thomason's theorem and a non-completed variant. 
\end{abstract}

\maketitle

\section{Introduction}

In \cite{SegalGamma}, Segal described a functor from (small) symmetric
monoidal categories to infinite loop spaces, or equivalently,
connective spectra.  This functor is often called the $K$-theory
functor: When applied to the symmetric monoidal category
of finite rank projective modules over a ring $R$, the resulting
spectrum is Quillen's algebraic $K$-theory of $R$. A natural question
is then which connective spectra arise as the $K$-theory of symmetric
monoidal categories?  Thomason answered this question in
\cite{ThomasonSymMon}, showing that every connective spectrum is the
$K$-theory of a symmetric monoidal category; moreover, he showed that
the $K$-theory functor is an equivalence between an appropriately defined
stable homotopy category of symmetric monoidal categories and the
stable homotopy category of connective spectra.

This paper provides a new proof of Thomason's theorem by constructing
a new homotopy inverse to Segal's $K$-theory functor.  As a model for
the category of infinite loop 
spaces, we work with $\Gamma$-spaces, following the usual conventions
of~\cite{Anderson,BousfieldFriedlander}: We understand a
$\Gamma$-space to be a functor $X$ from $\Gamma^{\op}$ (finite based
sets) to based simplicial sets such that $X(\fpz)=*$.  A $\Gamma$-space has
an associated spectrum \cite[\S1]{SegalGamma} (or
\cite[\S4]{BousfieldFriedlander} with these conventions), and a
map of $\Gamma$-spaces $X\to Y$ is called
a \term{stable equivalence} when it induces a stable equivalence of
the associated spectra.  We understand the stable homotopy category of
$\Gamma$-spaces to be the homotopy category obtained by formally
inverting the stable equivalences.  The foundational theorem of Segal
\cite[3.4]{SegalGamma}, \cite[5.8]{BousfieldFriedlander} is that the
stable homotopy category of $\Gamma$-spaces is equivalent to the stable
category of connective spectra.

On the other side, the category of small symmetric monoidal categories
admits a number of variants, all of which have equivalent stable
homotopy categories.  We discuss some of these variants in
Section~\ref{secrevk}
below.  For definiteness, we state the main theorem in terms of the
category of small permutative categories and strict maps: The objects
are the small permutative categories, i.e., those symmetric monoidal
categories with strictly associative and unital product, and the maps
are the functors that strictly preserve the product, unit, and
symmetry.  Segal \cite[\S2]{SegalGamma} constructed $K$-theory as a
composite functor $K^{\Seg}=N\circ \aK^{\Seg}$ from permutative
categories to $\Gamma$-spaces, where $\aK^{\Seg}$ is a functor from
permutative categories to $\Gamma$-categories, and $N$ is the nerve
construction 
applied objectwise to a $\Gamma$-category to obtain a $\Gamma$-space.  We
actually use a slightly different but weakly equivalent functor
$K=N\circ \aK$ described in Section~\ref{secrevk}.  A
\term{stable equivalence} of permutative categories is defined to be a
map that induces a stable equivalence on $K$-theory $\Gamma$-spaces.  (We
review an equivalent more intrinsic homological definition of stable
equivalence in Proposition~\ref{propQuil} below.)  We understand the stable
homotopy category 
of small permutative categories to be the homotopy category obtained
from the category of small permutative categories by formally
inverting the stable equivalences.

In Section~\ref{secconsp}, we construct a functor $P$ from
$\Gamma$-spaces to small permutative categories. Like $K$, we
construct $P$ as 
a composite functor $P=\aP\circ \simp$, with $\aP$ a functor from
$\Gamma$-categories to permutative categories and $\simp$ a functor from
simplicial sets to categories applied objectwise.  The functor $\simp$
is the left adjoint of the Quillen equivalence between the category of
small categories and the category of simplicial sets from
\cite{nonmodel,ThomasonCatModel}; the right adjoint is $\Ex^{2}N$, where $\Ex$
is Kan's right adjoint to the subdivision functor $\Sd$. As we
review in Section~\ref{secrevgamm}, we have natural transformations
\begin{equation}\label{eqnatns}
N\simp X\from \Sd^{2}X\to X
\qquad \text{and}\qquad 
\simp N \aX\to \aX
\end{equation}
which are always weak equivalences,
where we understand a weak  
equivalence of categories as a functor that induces a weak
equivalences on nerves.  The functor $\aP$ from $\Gamma$-categories to
permutative categories is a certain Grothendieck construction
(homotopy colimit)
\[
\aP(\aX)=\aA\groth A\aX, 
\]
we describe in detail in Section~\ref{secconsp}.  In brief, $\aA$ is a category
whose objects are the sequences of positive integers
$\aom=(m_{1},\dotsc,m_{r})$ including the empty sequence, and whose
morphisms are generated by permuting the sequence, maps of finite
(unbased) sets, and partitioning; for a $\Gamma$-category $X$, we get a
(strict) functor $A\aX$ from $\aA$ to the category of small categories
satisfying
\[
A\aX(m_{1},\dotsc,m_{r})=\aX(\fpm_{1})\times \dotsb \times \aX(\fpm_{r})
\]
and $A\aX()=\aX(\fpz)=*$ (the category with a unique object $*$ and
unique morphism).  The concatenation of sequences induces the
permutative product on $\aP\aX$ with $*$ in $A\aX()$ as the unit.  In
Section~\ref{secconsp}, 
we construct a natural transformation of permutative categories and
natural transformations of $\Gamma$-categories
\begin{equation}\label{eqnatcat}
\aP\aK\aC\to \aC \qquad \text{and}\qquad \aX\from \aW\aX\to \aK\aP\aX,
\end{equation}
where $\aW$ is a certain functor from $\Gamma$-categories to itself
(Definition~\ref{defW}).  In Section~\ref{secpf}, we show that these natural
transformations 
are natural stable equivalences, which then proves the following
theorem, the main theorem of the paper. 

\begin{thm}\label{maincat}
The functor $P$ from $\Gamma$-spaces to small permutative
categories preserves stable equivalences.  It induces an equivalence
between the stable homotopy category of $\Gamma$-spaces and the stable
homotopy category of permutative categories, inverse to Segal's
$K$-theory functor.
\end{thm}

The arguments actually prove a ``non-group-completed'' version of this
theorem.  To explain this, recall that a $\Gamma$-space $X$ is called
\term{special} \cite[p.~95]{BousfieldFriedlander} when the canonical
map $X(\fps{a}\vee \fps{b})\to X(\fps{a})\times X(\fps{b})$ is a weak
equivalence for any finite based sets $\fps{a}$ and $\fps{b}$; we
define a special $\Gamma$-category analogously.  Note that because of the
weak equivalences in~\eqref{eqnatns}, a $\Gamma$-category $\aX$ is
special if and only if the $\Gamma$-space $N\aX$ is special, and a $\Gamma$-space $X$ is special if and only if the $\Gamma$-category $\simp X$ is
special.  For special $\Gamma$-spaces, the associated spectrum is an
$\Omega$-spectrum after the zeroth space \cite[1.4]{SegalGamma}; the
associated infinite loop space is the group completion of $X(\fpo)$.
For any permutative category $\aC$, $K\aC$ is a special $\Gamma$-space
and $\aK\aC$ is a special $\Gamma$-category.  We show in
Corollary~\ref{corPK} that 
the natural transformation $\aP\aK\aC\to \aC$ of~\eqref{eqnatcat} is a
weak equivalence for any permutative category $\aC$, and we show in
Theorem~\ref{thmW} and Corollary~\ref{corKP} that the natural
transformations $W\aX\to\aX$ and 
$W\aX\to \aK\aP\aX$ of~\eqref{eqnatcat} are (objectwise) weak
equivalences for any special $\Gamma$-category $\aX$.  We obtain the
following theorem. 

\begin{thm}\label{thmuncomp}
The following homotopy categories are equivalent:
\begin{enumerate}
\item The homotopy category obtained from the category of small
permutative categories by inverting the weak equivalences.
\item The homotopy category obtained from the subcategory of special
$\Gamma$-spaces by inverting the objectwise weak equivalences.
\end{enumerate}
\end{thm}

Theorem~\ref{maincat} implies that for an arbitrary $\Gamma$-space $X$,
the $\Gamma$-space $KP X$ is a special $\Gamma$-space stably equivalent to
$X$.  A construction analogous to $\aP$ on the simplicial set level
produces such a special $\Gamma$-space more directly: For a $\Gamma$-space
$X$, we get a functor $AX$ from $\aA$ to based simplicial sets with
\[
AX(m_{1},\dotsc,m_{r})=X(\fpm_{1})\times \dotsb \times X(\fpm_{r})
\]
and $AX()=X(\fpz)=*$.  Define
\[
EX=\hocolim_{\aA}AX.
\]
Using the $\Gamma$-spaces $X(\fpn\sma(-))$, we obtain a $\Gamma$-space $E^{\Gamma}X$,
\[
E^{\Gamma}X(\fpn)=E(X(\fpn\sma(-)))/N\aA,
\]
with $EX\to E^{\Gamma}X(\fpo)$ a weak equivalence.
The inclusion of $X(\fpo)$ as $AX(1)$ provides a natural
transformation of simplicial sets $X(\fpo)\to EX$ and of $\Gamma$-spaces $X\to E^{\Gamma}X$.  In Section~\ref{secss}, we prove the following
theorems about these constructions. 

\begin{thm}\label{maingamma}
For any $\Gamma$-space $X$, the $\Gamma$-space $E^{\Gamma}X$ is special and
the natural map $X\to E^{\Gamma}X$ is a stable equivalence.   If $X$
is special, then the natural map $X\to E^{\Gamma}X$ is an objectwise
weak equivalence.
\end{thm}

\begin{thm}\label{maineinfty}
For any $\Gamma$-space $X$, the simplicial set $EX$ has the natural
structure of an $E_{\infty}$ space over the Barratt-Eccles operad (and
in particular the structure of a monoid).
\end{thm}

The previous two theorems functorially produce two additional
infinite loop spaces from the $\Gamma$-space $X$, the infinite loop
space of the spectrum associated to $E^{\Gamma}X$ and the group
completion of $EX$.  Since the map $X\to E^{\Gamma}X$ is a stable
equivalence, it induces a stable equivalence of the
associated spectra and hence the associated infinite loop spaces.  The
celebrated theorem of May and Thomason \cite{MayThomason} then
identifies the group completion of $EX$.

\begin{cor}
For any $\Gamma$-space $X$, the group completion of the $E_{\infty}$
space $EX$ is equivalent to the infinite loop space associated to $X$.
\end{cor}

As a consequence of Theorem~\ref{maingamma}, the
$\Gamma$-space $E^{\Gamma}X$ is homotopy initial among maps from $X$ to a
special $\Gamma$-space.  Theorem~\ref{maineinfty} then identifies
$EX\simeq E^{\Gamma}X(\fpo)$ as a reasonable candidate for the (non-completed)
$E_{\infty}$ space of $X$.  Motivated by Theorem~\ref{thmuncomp}, we
propose the following definition. 

\begin{defn}
We say a map of $\Gamma$-spaces $X\to Y$ is a \term{pre-stable equivalence}
when the map $EX\to EY$ is a weak equivalence.
\end{defn}

With this definition we obtain the equivalence of the last three
homotopy categories in the 
following theorem from the theorems above.  We have included the first
category for easy comparison with other non-completed theories of
$E_{\infty}$ spaces; we prove the equivalence in Section~\ref{secss}.

\begin{thm}\label{mainfouruncomp}
The following homotopy categories are equivalent:
\begin{enumerate}
\item\label{enmfa} The homotopy category obtained from the category of
$E_{\infty}$ spaces over the Barratt-Eccles operad (in simplicial
sets) by inverting the weak equivalences.
\item\label{enmfb} The homotopy category obtained from the category of
$\Gamma$-spaces by inverting the pre-stable equivalences.
\item\label{enmfc} The homotopy category obtained from the subcategory
of special $\Gamma$-spaces by inverting the objectwise weak equivalences.
\item\label{enmfd} The homotopy category obtained from the category of
small permutative categories by inverting the weak
equivalences.
\end{enumerate}
\end{thm}

The previous theorem provides a homotopy theory for permutative
categories and $\Gamma$-spaces before group completion, which now
allows the construction of ``spectral monoid rings'' associated to
$\Gamma$-spaces.  For a topological monoid $M$, the suspension
spectrum $\Sigma^{\infty}_{+}M$ has the structure of an associative
$S$-algebra ($A_{\infty}$ ring spectrum) with $M$ providing the
multiplicative structure.  The spectral monoid ring is a stable
homotopy theory refinement of the monoid ring $\bZ[\pi_{0}M]$, which
is $\pi_{0}\Sigma^{\infty}_{+}M$ or $\pi_{0}^{S}M$.  For a
$\Gamma$-space $X$, we can use $EX$ in place of $M$ and
$\Sigma^{\infty}_{+}EX$ is an $E_{\infty}$ ring spectrum with the
addition on $EX$ providing the multiplication on
$\Sigma^{\infty}_{+}EX$.  The spectral group ring of the associated
infinite loop space, $\Sigma^{\infty}_{+}\Omega^{\infty} X$, is the
localization of $\Sigma^{\infty}_{+}EX$ with respect the
multiplicative monoid $\pi_{0}EX\subset \pi_{0}^{S}EX$.  Spectral
monoid rings and algebras arise in the construction of twisted
generalized cohomology theories (as explained, for example, in
\cite[2.5]{units} and \cite{twists}), and the localization
$\Sigma^{\infty}_{+}EX\to\Sigma^{\infty}_{+}\Omega^{\infty} X$,
specifically, plays a role in current work in extending notions
of log geometry to derived algebraic geometry and stable homotopy
theory (see the lecture notes by Rognes on log geometry available at
\cite{loglects}).

\subsection*{Acknowledgments}
This paper owes an obvious debt to the author's collaborative work
with A.~D.~Elmendorf \cite{ElmendorfMandell,ElmendorfMandell2}; the
author thanks A.~D.~Elmendorf 
for many useful conversations and remarks.  The author thanks
Andrew J.~Blumberg for all his help. 

\section{Review of $\Gamma$-categories and $\Gamma$-spaces}
\label{secrevgamm}

This section briefly reviews the equivalence between the homotopy theory
of $\Gamma$-spaces and of $\Gamma$-categories.  We begin by introducing the
notation used throughout the paper.

\begin{notn}
We denote by $\fun$ the finite set $\{1,\dotsc,n\}$ and $\fpn$ the
finite based set $\{0,1,\dotsc,n\}$, with zero as base-point.  We
write $\aN$ for the category with objects the finite sets $\fun$ for
$n\geq 0$ (with $\fuz$ the empty set) and morphisms the maps of sets.
We write $\aF$ for the 
category with objects the finite based sets $\fpn$ for $n\geq 0$ and
morphisms the based maps of based sets.
\end{notn}

We typically regard a $\Gamma$-space or $\Gamma$-category as a functor from
$\aF$ to simplicial sets or categories rather than from the whole category
of finite based sets.

\begin{defn}
A \term{$\Gamma$-space} is a functor $X$ from $\aF$ to simplicial sets with
$X(\fpz)=*$.  A map of $\Gamma$-spaces is a natural transformation of
functors from $\aF$ to simplicial sets. A \term{$\Gamma$-category} is a functor $\aX$
from $\aF$ to the category of small categories with $\aX(\fpz)=*$, the
category with the unique object $*$ and the unique morphism $\id_{*}$.
A map of $\Gamma$-categories is a natural transformation of functors
from $\aF$ to small categories.
\end{defn}

We emphasize that $\aX$ must be a strict functor to small categories:
For $\phi \colon \fpm\to \fpn$ and $\psi\colon \fpn\to \fpp$, 
the functors $\aX(\psi \circ \phi)$ and $\aX(\psi)\circ \aX(\phi)$
must be equal (and not just naturally isomorphic).  A map of $\Gamma$-categories $f\colon \aX\to \aY$ consists of a sequence of functors
$f_{n}\colon \aX(\fpn)\to \aY(\fpn)$ such that for every map $\phi \colon
\fpm\to\fpn$ in $\aF$, the diagram
\[
\xymatrix{%
\aX(\fpm)\ar[r]^{f_{m}}\ar[d]_{\aX(\phi)}&\aY(\fpm)\ar[d]^{\aY(\phi)}\\
\aX(\fpn)\ar[r]_{f_{n}}&\aY(\fpn)
}
\]
commutes strictly.  In particular, applying the nerve functor
objectwise to a $\Gamma$-category then produces a $\Gamma$-space.  We use
this in defining the ``strict'' homotopy theory of $\Gamma$-categories.

\begin{defn}
A map $X\to Y$ of $\Gamma$-spaces is a \term{weak equivalence} if each
map $X(\fpn)\to Y(\fpn)$ is a weak equivalence of simplicial sets.  A
map of $\Gamma$-categories $X\to Y$ is a \term{weak equivalence} if the
induced map $N\aX\to N\aY$ is a weak equivalence of $\Gamma$-spaces.
\end{defn}

More important than weak equivalence is the notion of stable
equivalence, which for our purposes is best understood in terms of
very special $\Gamma$-spaces.   A $\Gamma$-space $X$ is \term{special} when
for each $n$ the canonical map 
\[
X(\fpn)\to X(\fpo)\times \dotsb \times X(\fpo)=X(\fpo)^{\times n}
\]
is a weak equivalence.  This canonical map is induced by the
\term{indicator} maps $\fpn\to \fpo$ which send all but one of the
non-zero element of $\fpn$ to $0$.  For a special $\Gamma$-space,
$\pi_{0}X(\fpo)$ is an abelian monoid under the operation
\[
\pi_{0}X(\fpo) \times \pi_{0}X(\fpo) \iso \pi_{0}X(\fpt)
\to \pi_{0}X(\fpo)
\]
induced by the map $\fpt\to\fpo$ sending both non-basepoint elements
of $\fpt$ to the non-basepoint element of $\fpo$.  A special $\Gamma$-space is \term{very special} when the monoid $\pi_{0}X(\fpo)$ is a
group.  

\begin{defn}
A map of $\Gamma$-spaces $f\colon X\to Y$ is a \term{stable equivalence} when for
every very special $\Gamma$-space $Z$, the map $f^{*}\colon [Y,Z]\to
[X,Z]$ is a bijection, where $[-,-]$ denotes maps in the homotopy
category obtained by formally inverting the weak equivalences.  A map
of $\Gamma$-categories $\aX\to \aY$ is a \term{stable equivalence} when
the induced map $N\aX\to N\aY$ of $\Gamma$-spaces is a stable equivalence.
\end{defn}

Equivalently, a map of $\Gamma$-spaces is a stable equivalence if and
only if it induces a weak equivalence of associated spectra
\cite[5.1,5.8]{BousfieldFriedlander}. 

In order to compare the homotopy theory of $\Gamma$-spaces and
$\Gamma$-categories, we use the Fritsch-Latch-Thomason Quillen equivalence of the category of
simplicial sets and the category of small categories
\cite{nonmodel,ThomasonCatModel}.  We call a map in the category of small
categories a weak equivalence if it induces a weak equivalence on
nerves.  The nerve functor has a left adjoint ``categorization
functor'' $c$, which generally does not behave well homotopically.
However, $c\circ \Sd^{2}$ preserves weak equivalences, where
$\Sd^{2}=\Sd\circ \Sd$ is the second subdivision functor
\cite[\S7]{KanCSSC}.  The functor $\Ex^{2}N$ is right adjoint to
$c\Sd^{2}$, and for any simplicial set $X$ and any category
$\aC$, the unit and counit of the adjunction,
\[
X\to \Ex^{2}Nc\Sd^{2}X
\qquad \text{and}\qquad
c\Sd^{2}\Ex^{2}N\aC\to \aC,
\]
are always weak equivalences.  Since the natural map $X\to \Ex^{2}X$
is always a weak equivalence and the diagrams
\[
\xymatrix{%
&\Sd^{2}X\ar[r]\ar[d]_{\sim}&Nc\Sd^{2}X\ar[d]^{\sim}
&c\Sd^{2}N\aC\ar[d]^{\sim}\\
X\ar[r]^(.35){\sim}\ar@/_2em/[rr]_{\sim}&\Ex^{2}\Sd^{2}X\ar[r]&\Ex^{2}Nc\Sd^{2}X
&c\Sd^{2}\Ex^{2}N\aC\ar[r]^(.7){\sim}&\aC
}
\vrule height0pt width0pt depth2em
\]
commute, we have natural weak equivalences
\begin{equation}\label{eqnc}
Nc\Sd^{2}X \from \Sd^{2}X\to X
\qquad \text{and}\qquad
c\Sd^{2}N\aC\to \aC.
\end{equation}
The functor $\Sd^{2}$ takes the one-point simplicial set $*$ to an
isomorphic simplicial set; replacing $\Sd^{2}$ by an isomorphic
functor if necessary, we can arrange that $\Sd^{2}*=*$.

\begin{defn}
Let $\simp$ be the functor from $\Gamma$-spaces to $\Gamma$-categories
obtained by applying $c\Sd^{2}$ objectwise.
\end{defn}

We then obtain the natural weak equivalences of $\Gamma$-spaces and
$\Gamma$-categories~\eqref{eqnatns} from~\eqref{eqnc}. Inverting weak 
equivalences or stable equivalences, we get equivalences of
homotopy categories.

\begin{prop}\label{propNsimp}
The functors $N$ and $\simp$ induce inverse equivalences between the
homotopy categories of $\Gamma$-spaces and $\Gamma$-categories
obtained by inverting the weak equivalences.
\end{prop}

\begin{prop}\label{propNsimpS}
The functors $N$ and $\simp$ induce inverse equivalences between the
homotopy categories of $\Gamma$-spaces and $\Gamma$-categories
obtained by inverting the stable equivalences.
\end{prop}

Since both the weak equivalences and stable equivalences of $\Gamma$-spaces provide the weak equivalences in model structures (see, for
example, \cite{BousfieldFriedlander}), the homotopy categories in the
previous propositions are isomorphic to categories with small hom sets.

\section{Review of the $K$-theory functor}
\label{secrevk}

This section reviews Segal's $K$-theory functor from symmetric
monoidal categories to $\Gamma$-spaces and some variants of this functor.
All the material in this section is well-known to experts, and most
can be found in
\cite{ElmendorfMandell,MayPerm,MaySpecPerm,ThomasonSymMon}.  We
include it here to refer to specific details, for completeness, and to
make this paper more self-contained.

For a small symmetric monoidal category $\aC$, we typically
denote the symmetric monoidal product as $\boxbin$ and the unit as
$u$.  We construct a $\Gamma$-category $\aK\aC$ as follows.

\begin{cons}\label{consK}
Let $\aK\aC(\fpz)=*$ the category with a unique object $*$ and the
identity map.  For $\fpn$ in $\aF$ with $n>0$, we define the category
$\aK\aC(\fpn)$ to have as objects the collections $(x_{I},f_{I,J})$
where 
\begin{itemize}
\item For each subset $I$ of $\fun=\{1,\dotsc,n\}$,
$x_{I}$ is an object of $\aC$, and 
\item For each pair of disjoint subsets $I,J$ of $\fun$, 
\[
f_{I,J}\colon x_{I\cup J}\to x_{I}\boxbin x_{J}
\]
is a map in $\aC$
\end{itemize}
such that
\begin{itemize}
\item When $I$ is the empty set $\emptyset$, $x_{I}=u$ and $f_{I,J}$
is the inverse of the unit isomorphism.
\item $f_{I,J}=\gamma \circ f_{J,I}$ where $\gamma$ is the symmetry
isomorphism $x_{J}\boxbin x_{I}\iso x_{I}\boxbin x_{J}$.
\item Whenever $I_{1}$, $I_{2}$, and $I_{3}$ are mutually disjoint,
the diagram
\[
\xymatrix@C+3pc{%
x_{I_{1}\cup I_{2}\cup I_{3}}\ar[r]^{f_{I_{1},I_{2}\cup I_{3}}}
\ar[d]_{f_{I_{1}\cup I_{2},I_{3}}}
&x_{I_{1}}\boxbin x_{I_{2}\cup I_{3}}\ar[d]^{\id\boxbin f_{I_{2},I_{3}}}\\
x_{I_{1}\cup I_{2}}\boxbin x_{I_{3}}\ar[r]_{f_{I_{1},I_{2}}\boxbin\id}
&x_{I_{1}}\boxbin x_{I_{2}}\boxbin x_{I_{3}}
}
\]
commutes (where we have omitted notation for the associativity isomorphism in $\aC$).
We write $f_{I_{1},I_{2},I_{3}}$ for the common
composite into a fixed association.
\end{itemize}
A morphism $g$ in $\aK\aC(\fpn)$ from $(x_{I},f_{I,J})$ to
$(x'_{I},f'_{I,J})$ consists of maps $h_{I}\colon x_{I}\to x'_{I}$ in
$\aC$ for all $I$
such that $h_{\emptyset}$ is the identity and the diagram
\[
\xymatrix{%
x_{I\cup J}\ar[r]^{h_{I\cup J}}\ar[d]_{f_{I,J}}&x'_{I\cup J}\ar[d]^{f'_{I,J}}\\
x_{I}\boxbin x_{J}\ar[r]_{h_{I}\boxbin h_{J}}&x'_{I}\boxbin x'_{J}
}
\]
commutes for all disjoint $I$, $J$.
\end{cons}

The categories $\aK\aC(\fpn)$ assemble into a $\Gamma$-category as
follows.   For
$\phi \colon \fpm\to \fpn$ in $\aF$ and $X=(x_{I},f_{I,J})$ in
$\aK\aC(\fpm)$,
define $\phi_{*}X=Y=(y_{I},g_{I,J})$ where
\[
y_{I} = x_{\phi^{-1}(I)}
\qquad \text{and}\qquad 
g_{I,J}=f_{\phi^{-1}(I),\phi^{-1}(J)}
\]
(replacing $*$ with $u$ or vice-versa if $\fpm$ or $\fpn$ is $\fpz$),
and likewise on maps.   We obtain a $\Gamma$-space by applying the nerve
functor to each category $\aK\aC(\fpn)$.

\begin{defn}
For a symmetric monoidal category $\aC$, $K\aC$ is the $\Gamma$-space
$K\aC(\fpn)=N\aK\aC(\fpn)$. 
\end{defn}

In terms of functoriality, $\aK$ is obviously functorial in
\term{strict} maps of symmetric monoidal categories, i.e., functors
$F\colon \aC\to \aD$ that strictly preserve the product $\boxbin$,
the unit object and isomorphism, and the associativity and symmetry
isomorphisms.  In fact, $\aK$ extends to a functor on the
\term{strictly unital op-lax} maps: An \term{op-lax} map of
symmetric monoidal categories consists of a functor $F\colon \aC\to
\aD$, a natural transformation
\[
\lambda \colon F(x\boxbin y)\to F(x)\boxbin F(y),
\]
and a natural transformation $\epsilon \colon F(u)\to u$
such that the following unit, symmetry, and
associativity diagrams commute,
\begin{gather*}
\xymatrix{%
F(u\boxbin y)\ar[d]_{F(\eta)}\ar[r]^{\lambda}
&F(u)\boxbin F(y)\ar[d]^{\epsilon}
&F(x\boxbin y)\ar[r]^{F(\gamma)}\ar[d]_{\lambda}
&F(y\boxbin x)\ar[d]^{\lambda}\\
F(y)
&u\boxbin F(y)\ar[l]^{\eta}
&F(x)\boxbin F(y)\ar[r]_{\gamma}
&F(y)\boxbin F(x)
}\\
\xymatrix@C-2pc{%
&F(x\boxbin y)\boxbin F(z)\ar[dr]^{\lambda}\\
F((x\boxbin y)\boxbin z)\ar[d]_{F(\alpha)}\ar[ur]^{\lambda}
&&(F(x)\boxbin F(y))\boxbin F(z)\ar[d]^{\alpha}\\
F(x\boxbin (y\boxbin z))\ar[dr]_{\lambda}
&&F(x)\boxbin (F(y)\boxbin F(z))\\
&F(x)\boxbin F(y\boxbin z)\ar[ur]_{\lambda}\\
}
\end{gather*}
where $\eta$, $\gamma$, and $\alpha$ denote the unit,
symmetry and associativity isomorphisms, respectively. An op-lax map is
\term{strictly unital} when the unit map $\epsilon$ is the identity
(i.e., $F$ strictly preserves the unit object). A strictly unital
op-lax map 
$\aC\to \aD$ induces a map of $\Gamma$-categories $\aK\aC\to \aK\aD$
sending $(x_{I},f_{I,J})$ in $\aK\aC(\fpn)$ to $(F(x_{I}),\lambda
\circ F(f_{I,J}))$ in $\aK\aD(\fpn)$, and likewise for morphisms.

We will need the following additional structure in Section~\ref{secpf}.
Recall 
that an op-lax natural transformation $\nu \colon F\to G$ between
op-lax maps is a natural transformation such that the following
diagrams commute.
\[
\xymatrix{%
F(x\boxbin y)\ar[r]^{\nu}\ar[d]_{\lambda}
&G(x\boxbin y)\ar[d]^{\lambda}
&F(u)\ar[rr]^{\nu}\ar[dr]_{\epsilon}
&&G(u)\ar[dl]^{\epsilon}\\
F(x)\boxbin F(y)\ar[r]_{\nu \boxbin \nu}
&G(x)\boxbin G(y)
&&u
}
\]
An op-lax natural transformation between strictly unital op-lax maps
induces a natural transformation between the induced maps of
$\Gamma$-categories, compatible with the $\Gamma$-structure. We
summarize the 
discussion of the previous paragraphs in the following proposition.

\begin{prop}
$\aK$ and $K$ are functors from the category of small symmetric
monoidal categories and strictly unital op-lax maps to the category
of $\Gamma$-categories and the category of $\Gamma$-spaces, respectively.
An op-lax natural transformation induces a natural
transformation on $\aK$ and a homotopy on $K$ between the induced
maps. 
\end{prop}

In particular, by restricting to the subcategory consisting of the
permutative categories and the strict maps, we get the functors $\aK$
and $K$ in the statements of the theorems in the introduction.  These
functors admit several variants, which extend to different variants of
the category of symmetric monoidal categories.

\begin{var}\label{varSegal}
In the construction of $\aK$, we can require the maps $f_{I,J}$ to be
isomorphisms.  This is Segal's original $K$-theory functor as
described for example in \cite{MaySpecPerm}.  The natural domain of
this functor is the category of small symmetric monoidal categories
and strictly unital strong maps; these are the strictly unital op-lax
maps where $\lambda$ is an isomorphism.
\end{var}

\begin{var}\label{varlax}
In the construction of $\aK$, we can require the maps $f_{I,J}$ to go
the other direction, i.e.,
\[
f_{I,J}\colon x_{I}\boxbin x_{J} \to x_{I\cup J}.
\]
This is the functor called Segal $K$-theory in
\cite{ElmendorfMandell}; its natural domain is the category of small
symmetric monoidal categories 
and strictly unital lax maps.  A lax map $\aC\to \aD$ consists of a
functor and natural transformations
\[
\lambda \colon F(x)\boxbin F(y)\to F(x\boxbin y)
\qquad \text{and}\qquad
\epsilon \colon u\to F(u)
\]
making the evident unit, symmetry, and
associativity diagrams commute.  Put another way, $(F,\lambda,\epsilon)$
defines an op-lax map $\aC\to \aD$ if and only if
$(F^{\op},\lambda,\epsilon)$ defines a lax map $\aC^{\op}\to
\aD^{\op}$. 
\end{var}

We also have variants which loosen the unit condition, but the
constructions occur most naturally by way of 
functors between different categories of small symmetric monoidal
categories.  The inclusion of the category with strictly unital op-lax
maps into the category with op-lax maps has a left
adjoint $U$. Concretely, $U\aC$ has objects the objects of $\aC$ plus
a new disjoint object $v$.  Morphisms in
$U\aC$ between objects of $\aC$ are just the 
morphisms in $\aC$, and morphisms to and from $v$ are defined by
\[
U\aC(x,v)=\aC(x,u),\qquad U\aC(v,x)=\emptyset, \qquad U\aC(v,v)=\{\id_{v}\},
\]
for $x$ an object of $\aC$ and $u$ the unit in $\aC$. We obtain a
symmetric monoidal product on $U\aC$ from the symmetric monoidal
product on $\aC$ with $v$ chosen to be a strict unit (i.e., $v\boxbin
x=x$ for all $x$ in $U\aC$); the inclusion of $\aC$ in $U\aC$ is then
op-lax monoidal and the functor $U\aC\to \aC$ sending $v$ to $u$ is
strictly unital op-lax (in fact, strict). Then $\aK\circ U$ defines a
functor from the category of small symmetric monoidal categories and
op-lax maps to $\Gamma$-categories.

To compare these variants and to understand $K$, we
construct weak equivalences
\begin{equation}\label{eqKspec}
p_{n}\colon \aK\aC(\fpn)\to \aC\times \dotsb \times \aC=\aC^{\times n}.
\end{equation}
Define $p_{n}$ to be the functor sends the object $(x_{I},f_{I,J})$ of
$\aK\aC(n)$ to the object
$(x_{\{1\}},\dotsc,x_{\{n\}})$ of $\aC^{\times n}$, and likewise for
maps.  This functor 
has a right adjoint $q_{n}$ that sends $(y_{1},\dotsc,y_{n})$ to the
system $(x_{I},f_{I,J})$ with 
\[
x_{\{i_{1},\dotsc,i_{r}\}}=
(\dotsb (y_{i_{1}}\boxbin y_{i_{2}})\boxbin \dotsb)\boxbin y_{i_{r}}
\]
for $i_{1}<\dotsb<i_{r}$ and the maps $f_{I,J}$ induced by the
associativity and symmetry isomorphisms.  In the case $n=1$, these
are inverse isomorphisms of categories, and under these isomorphisms,
\eqref{eqKspec} is induced by the indicator maps $\fpn\to \fpo$.
Because an adjunction induces inverse homotopy equivalences on nerves,
this proves the following proposition.

\begin{prop}\label{propKspec}
For any small symmetric monoidal category $\aC$, $K\aC$ is a special
$\Gamma$-space with $K\aC(\fpo)$ isomorphic to $N\aC$.
\end{prop}

Similar observations apply to the variant functors above.  We have
natural transformations relating the strictly unital strong
construction to both the strictly unital lax and 
op-lax constructions.  It follows that the $K$-theory
$\Gamma$-spaces obtained are naturally weakly equivalent.  Likewise, the
constructions with the weakened units map to the constructions with
strict units.  Since the map $U\aC\to \aC$ induces a homotopy
equivalence on nerves, these natural transformations induce natural
weak equivalences of $\Gamma$-spaces.

Recall that we say that a functor between small categories is a
\term{weak equivalence} when it induces a weak equivalence on nerves.
As a consequence of the previous proposition, the $K$-theory functor
preserves weak equivalences.  As in the introduction, we say that a
strictly unital op-lax map of symmetric monoidal categories is a
\term{stable equivalence} if it induces a stable equivalence on
$K$-theory $\Gamma$-spaces, or equivalently, if it induces a weak
equivalence on the group completion of the nerves.  Quillen's
homological criterion to identify the group completion
\cite{SegalMcDuff} then applies to give an intrinsic characterization
of the stable equivalences.

\begin{prop}\label{propQuil}
A map of symmetric monoidal categories $\aC\to \aD$ is a stable
equivalence if and only if it induces an isomorphism of localized
homology rings
\[
H_{*}\aC[(\pi_{0}\aC)^{-1}]\to 
H_{*}\aD[(\pi_{0}\aD)^{-1}]
\]
obtained by inverting the multiplicative monoids $\pi_{0}\aC\subset
H_{0}\aC$ and $\pi_{0}\aD\subset H_{0}\aD$. 
\end{prop}

Although not needed in what follows, for completeness of
exposition, we offer the following well-known 
observation on the homotopy theory of the various categories of
symmetric monoidal categories.  Recall that we say that a functor
between small categories is a \term{weak equivalence} when it induces
a weak equivalence on nerves.  The following theorem can be proved
using the methods of \cite[4.2]{MayPerm} and
\cite[4.2]{ElmendorfMandell2}.

\begin{thm}\label{thmsymmoncomp}
The following homotopy categories are equivalent:
\begin{enumerate}
\item The homotopy category obtained from the category of small
permutative categories and strict maps by inverting the weak equivalences.
\item The homotopy category obtained from the category of small
symmetric monoidal categories and strict maps by inverting the weak equivalences.
\item The homotopy category obtained from the category of small
symmetric monoidal categories and strictly unital strong maps by
inverting the weak equivalences. 
\item The homotopy category obtained from the category of small
symmetric monoidal categories and strictly unital op-lax maps by
inverting the weak equivalences. 
\item The homotopy category obtained from the category of small
symmetric monoidal categories and strictly unital lax maps by
inverting the weak equivalences. 
\end{enumerate}
\end{thm}

\section{Construction of the inverse $K$-theory functor}
\label{secconsp}

In this section we construct the inverse $K$-theory functor $\aP$ as a
Grothendieck construction (or homotopy colimit) over a category $\aA$
described below.  We construct the natural transformations displayed
in~\eqref{eqnatcat} relating the composites $\aK\aP$ and $\aP\aK$ with
the identity.  This section contains only the constructions; we
postpone almost all homotopical analysis to the next section.

We begin with the construction of the category $\aA$.  As indicated in
the introduction, we define the objects of $\aA$ to consist of the
sequences of positive integers $(n_{1},\dotsc,n_{s})$ for all $s\geq
0$, with $s=0$ corresponding to the empty sequence $()$.  We think of
each $n_{i}$ as the finite (unbased) set $\fun_{i}$, and we define the
maps in $\aA$ to be the maps generated by maps of finite sets,
permutations in the sequence, and partitioning $\fun_{i}$ into
subsets.  We make this precise in the following definition. 

\begin{defn}\label{defA}
For $\aom=(m_{1},\dotsc,m_{r})$ and
$\aon=(n_{1},\dotsc,n_{s})$ with $r,s>0$, we define the morphisms
$\aA(\aom,\aon)$ to be the subset of the maps of finite (unbased) sets
\[
\fum_{1}\amalg  \dotsb \amalg  \fum_{r}\to \fun_{1}\amalg  \dotsb \amalg  \fun_{s}
\]
satisfying the property that the inverse image of each subset
$\fun_{j}$ is either empty or contained in a single $\fum_{i}$ (depending on $j$).
For the object $()$, we define $\aA((),\aon)$ consist of a single point
for all $\aon$ in $\aA$ and we define $\aA(\aom,())$ to be empty for
$\aom\neq()$. 
\end{defn}

For a $\Gamma$-category $\aX$, let $A\aX()=\aX(\fpz)$ and
\[
A\aX(n_{1},\dotsc,n_{s})=\aX(\fpn_{1})\times \dotsb \times \aX(\fpn_{s}).
\]
For a map $\phi \colon \aom\to \aon$ in $\aA$, define
\[
A\phi \colon
\aX(\fpm_{1})\times \dotsb \times \aX(\fpm_{r})\to
\aX(\fpn_{1})\times \dotsb \times \aX(\fpn_{s})
\]
as follows.  If $s=0$, then $r=0$ and $\phi$ is the identity, and we
take $A\phi$ to be the identity.  If $r=0$ and $s>0$, we take $A\phi$
to be the map $\aX(\fpz)\to \aX(\fpn_{j})$ on each coordinate.  If
$r>0$, then by definition, for each $j$, the subset $\fun_{j}$ of
\[
\fun_{1}\amalg \dotsb \amalg \fun_{s}
\]
has inverse image either empty or contained in a single $\fum_{i}$ for
some $i$; if the inverse image is
non-empty, then $\phi$ restricts to a map of unbased sets
$\fum_{i}\to\fun_{j}$, which we extend to a map of based sets
$\fpm_{i}\to\fpn_{j}$ that is the identity on the basepoint $0$. 
In this case, we define $A\phi$ on the $j$-th coordinate to be the
composite of the projection
\[
\aX(\fpm_{1})\times \dotsb \times \aX(\fpm_{r})\to \aX(\fpm_{i})
\]
and the map $\aX(\fpm_{i})\to \aX(\fpn_{j})$ induced by the restriction of
$\phi$.  In the case when the inverse image of $\fun_{j}$ is empty, we
define $A\phi$ on the $j$-th coordinate to be the composite of the projection
\[
\aX(\fpm_{1})\times \dotsb \times \aX(\fpm_{r}) \to *=\aX(\fpz)
\]
and the map $\aX(\fpz)\to \aX(\fpn_{j})$.  An easy check gives the
following observation.

\begin{prop}\label{propAfunc}
$A$ is a functor from the category of small $\Gamma$-categories to
the category of functors from $\aA$ to the category of small categories.
\end{prop}

We can now define the functors $\aP$ and $P$, at least on the level of functors
to small categories.

\begin{defn}\label{defP}
Let $\aP\aX=\aA\groth A\aX$.  Let $P=\aP\circ \simp$.
\end{defn}

More concretely, the category $\aP\aX$ has as objects the disjoint
union of the objects of $A\aX(\aon)$ where $\aon$ varies over the
objects of $\aA$.  For $x\in A\aX(\aom)$ and $y\in A\aX(\aon)$, a map
in $\aP\aX$ from $x$ to $y$ consists of a map $\phi \colon \aom\to
\aon$ in $\aA$ together with a map $\phi_{*}x\to y$ in $A\aX(\aon)$,
where $\phi_{*}=A\phi$ is the functor $A\aX(\aom)\to A\aX(\aon)$ above.

\begin{var}\label{varPlax}
We can regard $A\aX$ as a contravariant functor on $\aA^{\op}$ and
form the contravariant Grothendieck construction $\aPlax\aX=\aA^{\op}\groth
A\aX$.  This has the same objects as $\aP\aX$ but for
$x\in A\aX(\aom)$ and $y\in A\aX(\aon)$, a map
in $\aPlax\aX$ from $x$ to $y$ consists of a map $\phi \colon \aon\to
\aom$ in $\aA$ together with a map $x\to \phi_{*}y$.  This functor is
better adapted to the category of symmetric monoidal categories and
strictly unital lax maps.  All results and constructions in this paper
admit analogues
for $\aPlax$, replacing ``op-lax'' with ``lax'' in the work below.
\end{var}

The category $\aA$ has the structure of a permutative category
under concatenation of sequences, with the empty sequence as the unit
and the symmetry morphisms induced by permuting elements in the sequences.
The category of small categories is symmetric monoidal under
cartesian product and the functor $A\aX\colon \aA\to \Cat$ associated
to a $\Gamma$-category $\aX$ is a strong symmetric monoidal functor.  For
formal reasons, then the Grothendieck construction $\aP\aX$ naturally obtains the
structure of a symmetric monoidal category; we
can describe this structure concretely as follows.  For any object 
$x$ in $A\aX(\aom)$, we can write $x=(x_{1},\dotsc,x_{r})$ for objects
$x_{i}$ in $\aX(\fpm_{i})$; then for $y$ in $A\aX(\aon)$, 
\[
x \boxbin y = (x_{1},\dotsc,x_{r},y_{1},\dotsc,y_{s}) 
\in \Ob A\aX(m_{1},\dotsc,m_{r},n_{1},\dotsc,n_{s}),
\]
where we understand the unique object of $A\aX()$ as a strict
unit.  The product on maps admits an analogous description.  This
concrete description makes it clear that $\aP\aX$ is in fact a
permutative category.  Moreover, a map $\aX\to \aY$ of $\Gamma$-categories induces a strict map of permutative categories $\aP\aX\to
\aP\aY$.   We obtain the following theorem.

\begin{thm}
$\aP$ defines a functor from the category of $\Gamma$-categories to the
category of permutative categories and strict maps.
\end{thm}

Next we construct the natural transformations of~\eqref{eqnatcat}.
Starting with a symmetric monoidal category $\aC$, we construct the
map $\aP\aK\aC\to \aC$ using the homotopy colimit property of the
Grothendieck construction.  Specifically, we construct functors
$\alpha_{\aom}$ from $A\aP\aK\aC(\aom)$ to $\aC$ and suitably compatible
natural transformations for the maps in $\aA$.

For each $\aom$ in $\aA$, define the functor
\[
\alpha_{\aom}\colon A\aK\aC(\aom)=\aK\aC(\fpm_{1})\times \dotsb \times \aK\aC(\fpm_{r})
\to \aC
\]
to take the object $\aob{X}=(X_{1},\dotsc,X_{r})$ to 
\[
(\dotsb(x^{1}_{\fum_{1}} \boxbin x^{2}_{\fum_{2}})\boxbin\dotsb)\boxbin x^{r}_{\fum_{r}}
\]
where $X_{i}=(x^{i}_{I},f_{I,J})$ for $I\subset \fum_{i}$, and
likewise for maps in $\aK\aC(\aom)$.  For $\aom=()$, we understand
$\alpha_{()}$ to include the category $A\aK\aC()=*$ in $\aC$ as the unit
$u$ and the identity on $u$.

For a map $\phi$ in $\aA$ from $\aom$ to $\aon$, define
\[
\alpha_{\phi}\colon \alpha_{\aom}(y_{1},\dotsc,y_{r})\to \alpha_{\aon}(\phi_{*}(y_{1},\dotsc,y_{r}))
\]
to be the map induced by the associativity, symmetry, and inverse unit
isomorphisms in $\aC$ and the maps $f^{i}_{I_{1},\dotsc,I_{k}}$ in $y_{i}$: if $\phi$ sends
$\fum_{i}$ into $\fun_{j_{1}},\dotsc,\fun_{j_{t}}$, then composing
maps $f^{i}_{I,J}$ in $y_{i}$ gives a well-defined map
\[
f_{I_{1},\dotsc,I_{t}}\colon x^{i}_{\fum_{i}}\to 
(\dotsb (x^{i}_{I_{1}}\boxbin
x^{i}_{I_{2}})\boxbin \dotsb )\boxbin x^{i}_{I_{t}}
\]
where $I_{k}$ is the subset of $\fum_{i}$ landing in
$\fun_{j_{k}}$.  The map $\alpha_{\phi}$ is a natural transformation of
functors from $\alpha_{\aom}$ to $\alpha_{\aon}\circ \phi_{*}$.  Moreover, given
a map $\psi$ from $\aon$ to $\aop$ in $\aA$, the following diagram
commutes
\[
\xymatrix{%
\alpha_{\aom} \aob{X}\ar[r]^{\alpha_{\phi}}\ar[d]_{\alpha_{\psi \circ \phi}}
&\alpha_{\aon}(\phi_{*}\aob{X})\ar[d]^{\alpha_{\psi}}\\
\alpha_{\aop}((\psi \circ \phi)_{*}\aob{X})
\ar@{{}{}{}}[r]|{\textstyle=}
&\alpha_{\aop}(\psi_{*}\phi_{*}\aob{X})
}
\]
for any $\aob{X}$ in $A\aP\aK\aC(\aom)$.

\begin{defn}
Let $\alpha \colon \aP\aK\aC\to \aC$ be the functor $\aA\groth AX\to\aK\aC$ that
sends $\aob{X}$ in $A\aK\aC(\aom)$ to $\alpha_{\aom}\aob{X}$ and sends the
map 
\[
\phi\colon \aom\to \aon, \quad f\colon \phi_{*}\aob{X}\to \aob{Y}
\]
to the map $\alpha_{\aon}(f)\circ \alpha_{\phi}$ in $\aC$.
\end{defn}

Examining the construction of $\alpha$ and the symmetric monoidal structure
on $\aP\aK\aC$, we obtain the following theorem.

\begin{thm}\label{thmPK}
The functor $\alpha \colon \aP\aK\aC\to \aC$ satisfies the following properties.
\begin{enumerate}
\item $\alpha$ is a strictly unital strong map of symmetric monoidal categories.
\item\label{ennatnat} $\alpha$ is natural up to natural transformation
in strictly unital op-lax maps.
\item $\alpha $ is a strict map when $\aC$ is a permutative category.
\item\label{ennatstr} $\alpha$ is natural in strict maps.
\end{enumerate}
\end{thm}

The meaning of~\enref{ennatnat} and~\enref{ennatstr} is that for a
strictly unital op-lax map $F\colon \aC\to \aD$, the diagram
\[
\xymatrix{%
\aP\aK\aC\ar[r]\ar[d]_{\alpha}&\aP\aK\aD\ar[d]^{\alpha}\\
\aC\ar[r]&\aD
}
\]
commutes up to natural transformation, namely, the natural transformation
\[
\lambda \colon F(x_{1}\boxbin \dotsb \boxbin x_{r})\to F(x_{1})\boxbin
\dotsb \boxbin F(x_{r})
\]
which is part of the structure of the op-lax map.  When $\aC\to\aD$ is
a strict map, the diagram commutes strictly (the natural
transformation is the identity). 

For the remaining natural transformation in \eqref{eqnatcat}, note that for a
$\Gamma$-category $\aX$, we have a canonical inclusion $\iota \colon \aX(\fpn)\to
\aK\aP\aX(\fpn)$ sending an object $x$ in $\aX(\fpn)$ to the object
$\iota x=(x_{I},f_{I,J})$ in $\aK\aP\aX(\fpn)$ with 
\[
x_{I}=\pi^{I}_{*}(x) \in \aX(\fpm)=A\aX(m)
\]
for $m=|I|$, $I=\{i_{1},\dotsc,i_{m}\}$ with $i_{1}<\dotsb <i_{m}$,
and $\pi^{I}\colon \fpn\to \fpm$ the map that sends $i_{k}$ to $k$ and
every other element of $\fpn$ to $0$.  The map
\[
f_{I,J}\colon x_{I\cup J}\to x_{I}\boxbin x_{J}=(x_{I},x_{J})\in A\aX(|I|,|J|)
\]
is induced by the map $(|I\cup J|)\to (|I|,|J|)$ in $\aA$
corresponding to the partition of the ordered set $I\cup J$ into $I$
and $J$. This does not fit together into a map of $\Gamma$-categories:
For a map $\phi \colon \fpm\to\fpn$ in $\aF$, 
\[
\iota (\phi_{*}x)_{I}=\pi^{I}_{*}(\phi_{*}x),
\qquad \text{but} \qquad 
\phi_{*}(\iota x)_{I}=\pi^{\phi^{-1}I}_{*}(x).
\]
Writing $\phi'$ for the map in $\aN$ corresponding to the restriction
of $\phi$ to the map $\phi^{-1}I\to I$, then
\[
\pi^{I}\circ \phi = \phi' \circ \pi^{\phi^{-1}I}
\]
in $\aF$.  We can interpret $\phi'$ as a map
\[
\pi^{\phi^{-1}I}_{*}(x)\to \pi^{I}_{*}(\phi_{*}x)
\]
in $\aP\aX$.  These maps in turn assemble to a map
\[
\omega_{\phi} \colon \phi_{*}\iota x \to \iota \phi_{*}x
\]
in $\aK\aP\aX$, natural in $x$.

\begin{defn}\label{defW}
For a $\Gamma$-category $\aX$, let $\aW\aX(\fpn)$ be the category whose
objects consist of triples $(y,x,g)$ with $y$ an object of
$\aK\aP(\fpn)$, $x$ an object of $\aX(\fpn)$ and $g\colon y\to \iota
x$ a map in $\aK\aP(\fpn)$.  The morphisms of $\aW\aX(\fpn)$ are the
commuting diagrams.  For $\phi \colon \fpm\to \fpn$ in $\aF$, define 
\[
\aW\aX(\phi)\colon \aW\aX(\fpm)\to \aW\aX(\fpn)
\]
to be the functor that takes $(y,x,g)$ to
$(\phi_{*}y,\phi_{*}x,\omega_{\phi}\circ \phi_{*}g)$. 
\end{defn}

We note for later use that for $(y,x,g)$ an object in $\aW\aX(\fpn)$,
writing $y=(y_{I},f_{I,J})$ with $y_{I}$ in $\aP\aX$, we must have
each $y_{I}$ in $A\aX()$ or $A\aX(m_{I})$ for some $m_{I}$.  This is
because $\iota x_{I}$ is in $A\aX(n)$ and maps in $\aA$ cannot
decrease the length of the sequence.

We claim that the categories $\aW\aX(\fpn)$ and functors
$\aW\aX(\phi)$ assemble into a
$\Gamma$-category. For $\phi \colon \fpm\to \fpn$, $\psi
\colon \fpn\to \fpp$, and 
$I\subset \fup$, write
\[
\pi^{I}\circ \psi \circ \phi 
= \psi' \circ \pi^{\psi^{-1}I}\circ \phi 
= \psi'\circ \phi' \circ \pi^{\phi^{-1}(\psi^{-1}I)}
\]
as above, with $\psi'\colon \psi^{-1}I\to I$ and $\phi'\colon
\phi^{-1}(\psi^{-1}I)\to \psi^{-1}I$ the restrictions of $\psi$ and
$\phi$, using the natural order on $I\subset \fup$,
$\psi^{-1}I\subset \fun$, and $\phi^{-1}(\psi^{-1}I)\subset \fum$ to
view these as maps in $\aF$.  Then $\psi'\circ \phi'\colon (\psi \circ
\phi)^{-1}I\to I$ is the restriction of $\psi \circ \phi$; this is the
check required to see that the diagram

\[
\xymatrix@C-2pc{%
\pi^{(\psi \circ \phi)^{-1}I}_{*}x\ar[rr]\ar[dr]
&&\pi^{\psi^{-1}I}_{*}(\phi_{*}x)\ar[dl]\\
&\pi^{I}_{*}((\psi \circ \phi)_{*}x)
}
\]
in $\aP\aX$ commutes.  Examination of the structure maps $f_{I,J}$ in
$\iota x$ shows that
the diagram
\[
\xymatrix@C-2pc{%
(\psi \circ \phi)_{*}\iota x\ar[rr]^{\psi_{*}\omega_{\phi}}
\ar[dr]_{\omega_{\psi \circ \phi}}
&&\psi_{*}\iota (\phi_{*}x)\ar[dl]^{\omega_{\psi}}\\
&\iota ((\psi \circ \phi)_{*}x)
}
\]
in $\aK\aP\aX$ commutes.  This proves the following theorem.

\begin{thm}\label{thmconsW}
The maps $\aW\aX(\phi)$ above make $\aW\aX$ into a $\Gamma$-category.
\end{thm}

Since $\aW\aX$ is natural in maps of $\Gamma$-categories $\aX$, we can
regard $\aW$ as an endofunctor on  $\Gamma$-categories.  By
construction, the forgetful functors $\omega \colon \aW\aX\to \aX$ and
$\upsilon \colon \aW\aX\to \aK\aP\aX$ are natural transformations of
endofunctors.  For fixed $\fpn$, the functor $\aW\aX(\fpn)\to
\aX(\fpn)$ is a left adjoint: The right adjoint sends $x$ in
$\aX(\fpn)$ to $(\iota x,x,\id_{\iota x})$ in $\aW\aX(\fpn)$.  It
follows that $\omega$ is always a weak equivalence of $\Gamma$-categories.
We summarize this in the following theorem.

\begin{thm}\label{thmW}
The maps
$\upsilon \colon \aW\aX\to \aK\aP\aX$ and $\omega \colon \aW\aX\to
\aX$ are natural transformations of endofunctors on $\Gamma$-categories, and $\omega$ is a weak equivalence for any $\aX$.
\end{thm}

\section{Proof of Theorems~\ref{maincat} and~\ref{thmuncomp}}
\label{secpf}

This section provides the homotopical analysis of the functors and
natural transformations constructed in the previous section.
This leads directly to the proof of the main theorem,
Theorem~\ref{maincat}, and its non-completed variant,
Theorem~\ref{thmuncomp}.

Most of the arguments hinge on the following lemma of
Thomason~\cite{ThomasonHocolim}: 

\begin{lem}[Thomason]\label{lemThomason}
Let $\aA$ be a small category and $F$ a functor from $\aA$ to the
category of small categories.  There is a natural weak equivalence of
simplicial sets
\[
\hocolim_{\aA}NF\to N(\aA\groth F).
\]
\end{lem}

The natural transformation is easy to describe.
We write an object of $\aA\groth F$ as $(\aon,x)$ with $\aon$ an object of
$\aA$ and $x$ an object of $F\aon$, and we write a map in $\aA\groth F$
as $(\phi,f)\colon (\aon,x)\to (\aop,y)$ where $\phi\colon \aon\to
\aop$ is a map 
in $\aA$ and $f\colon F(\phi)(x)\to y$ is a map in $F\aop$.  Then a
$q$-simplex of the nerve $N(\aA\groth F)$ is a sequence of $q$
composable maps
\[
\xymatrix{%
(\aon_{0},x_{0})\ar[r]^{(\phi_{1},f_{1})}
&(\aon_{1},x_{1})\ar[r]^(.6){(\phi_{2},f_{2})}
&\dotsb \ar[r]^(.35){(\phi_{q},f_{q})}
&(\aon_{q},x_{q}).
}
\]
Likewise, a $q$-simplex in the homotopy colimit consists of a sequence
of $q$ composable maps in $\aA$ together with $q$ composable maps in $F(\aon_{0})$:
\[
\xymatrix@R-1pc{%
\aon_{0}\ar[r]^{\phi_{1}}
&\aon_{1}\ar[r]^{\phi_{2}}
&\dotsb \ar[r]^{\phi_{q}}
&\aon_{q}\\
x_{0}\ar[r]^{f_{1}}
&x_{1}\ar[r]^{f_{2}}
&\dotsb \ar[r]^{f_{q}}
&x_{q}.
}
\]
The natural transformation sends this simplex of the homotopy colimit
to the simplex 
\[
\xymatrix{%
(\aon_{0},x_{0})\ar[r]^{(\phi_{1},f_{1})}
&(\aon_{1},x'_{1})\ar[r]^(.6){(\phi_{2},f'_{2})}
&\dotsb \ar[r]^(.35){(\phi_{q},f'_{q})}
&(\aon_{q},x'_{q}),
}
\]
where 
$x'_{k}=F(\phi_{k,\dotsc,1})(x_{k})$ and
$f'_{k}=F(\phi_{k-1,\dotsc,1})(f_{k})$
for $\phi_{k,\dotsc,1}=\phi_{k}\circ \dotsb \circ \phi_{1}$.  A
Quillen Theorem~A style argument proves that this map is a weak
equivalence \cite[\S1.2]{ThomasonHocolim}.

Applying Thomason's lemma to the Grothendieck construction $\aA\groth
A\aX$, we get the following immediate observation.

\begin{prop}
$\aP$ preserves weak equivalences.
\end{prop}

The following theorem provides the main homotopical result we need for
the remaining arguments in this section.

\begin{thm}\label{thmtech}
Let $\aX$ be a special $\Gamma$-category.  Then the inclusion of
$\aX(\fpo)$ in $\aP\aX$ is a weak equivalence.
\end{thm}

\begin{proof}
Recall that $\aN$ denotes the category with objects $\fun=\{1,\dotsc,n\}$ and
morphisms the maps of sets.  We have an inclusion $\eta \colon \aN\to
\aA$ sending $\fuz$ to $()$ and $\fun$ to $(n)$ for $n>0$.  We
have a functor $\epsilon \colon \aA\to \aN$ sending
$\aon=(n_{1},\dotsc,n_{s})$ to $\fun$ with $n=n_{1}+\dotsb+n_{s}$.
Let $B\aX$ be the functor from $\aA$ to small categories defined by
\[
B\aX(\aon)=\epsilon^{*}\aX(\aon)=\aX(\epsilon(\aon))=\aX(\fpn).
\]
Then the maps $\fpn\to \fpn_{j}$ coming from the partition of $n$ as
$\aon$ induce a natural transformation of functors $B\aX\to A\aX$.
The hypothesis that $\aX$ is special implies that this map is an
objectwise weak equivalence.  Now applying Thomason's lemma, it
suffices to show that the inclusion of $N\aX(\fpo)$ in
$\hocolim_{\aA}NB\aX$ is a weak equivalence.

Since $B\aX=\epsilon^{*}\aX$ as a functor on $\aA$ and
$\aX=\eta^{*}B\aX$ as a functor on $\aN$, we have canonical maps
\begin{equation}\label{eqtechthm}
\hocolim_{\aN}N\aX\to \hocolim_{\aA}NB\aX \to \hocolim_{\aN}N\aX
\end{equation}
induced by $\epsilon$ and $\eta$.
The composite map on $\hocolim_{\aN}N\aX$ is induced by $\epsilon \circ \eta=\Id_{\aN}$,
and is therefore the identity.  The 
composite map on $\hocolim_{\aA}NB\aX$ is induced by $\eta \circ
\epsilon$.  We have a natural transformation from $\eta \circ
\epsilon$ to the identity functor
on $\aA$ induced by the partition maps,
\[
\eta\circ \epsilon (\aon) = (n) \to (n_{1},\dotsc,n_{s})=\aon.
\]
Because
\[
B\aX(\eta \circ \epsilon(\aon))=\aX(\fpn)=B\aX(\aon),
\]
we get a homotopy from the
composite map on $\hocolim_{\aA}NB\aX$ to the identity.  In other
words, we have shown that the maps in~\eqref{eqtechthm} are inverse
homotopy equivalences.  Since $\fuo$ is the final object in $\aN$, the
inclusion of $N\aX(\fpo)$ in $\hocolim_{\aN}N\aX$ is a homotopy
equivalence, and it follows that the inclusion of $N\aX(\fpo)$ in
$\hocolim_{\aA}NB\aX$ is a homotopy equivalence.
\end{proof}

When $\aX=\aK\aC$ for a small symmetric monoidal category
$\aC$, we have the canonical isomorphism $\aK\aC(\fpo)\iso\aC$,
and the composite map 
\[
\aC\iso \aK\aC(\fpo)\to \aP\aK\aC\to \aC
\]
is the identity on $\aC$.  Since $\aK\aC$ is always a special $\Gamma$-category, we get the following corollary.

\begin{cor}\label{corPK}
The natural map $\alpha \colon \aP\aK\aC\to \aC$ is always a weak equivalence.
\end{cor}

We also get a comparison for $\aK\aP\aX$ when
$\aX$ is special.

\begin{cor}\label{corKP}
If $\aX$ is special, then $\upsilon \colon \aW\aX\to \aK\aP\aX$ is a
weak equivalence.
\end{cor}

\begin{proof}
Let $\aX$ be a special $\Gamma$-category.
The restriction of the map $\omega$ to the $\fpo$-cat\-e\-go\-ries,
$\aW\aX(\fpo)\to\aX(\fpo)$, is an equivalence of categories, and
the composite map 
\[
\aX(\fpo)\to \aW\aX(\fpo)\to \aK\aP\aX(\fpo)=\aP\aX
\]
is the map in the theorem, and therefore a weak equivalence.  It
follows that the restriction of $\upsilon$ to the $\fpo$-categories is a
weak equivalence. Since the map $\omega \colon \aW\aX\to \aX$ is a
weak equivalence, $\aW\aX$ is also a special $\Gamma$-category, and it
follows that $\upsilon$ is a weak equivalence.
\end{proof}

Together with Proposition~\ref{propNsimp} and Theorem~\ref{thmW},
Corollaries~\ref{corPK} and~\ref{corKP} prove Theorem~\ref{thmuncomp}.
To prove Theorem~\ref{maincat}, we need to see that the map $\upsilon$
is always a stable equivalence.  For this we use the following
technical lemma.

\begin{lem}\label{lemtechlem}
The diagram
\[
\xymatrix{%
\aW\aK\aP\aX\ar[r]^{\upsilon}\ar[dr]_{\omega}
&\aK\aP\aK\aP\aX\ar[d]_{\aK\alpha}
&\aK\aP\aW\aX\ar[l]_{\aK\aP\upsilon}\ar[dl]^{\aK\aP\omega}\\
&\aK\aP\aX
}
\]
commutes up to natural transformation of maps of $\Gamma$-categories. 
All maps in the diagram are weak equivalences.
\end{lem}

\begin{proof}
The weak equivalence statement follows from the diagram statement
since $\alpha$ and $\omega$ are always weak equivalences and $\aK$ and
$\aP$ preserve weak equivalences.  For the diagram statement, it
suffices to show that the diagrams 
\[
\xymatrix{%
\aW\aK\aC\ar[r]^{\upsilon}\ar[dr]_{\omega}
&\aK\aP\aK\aC\ar[d]^{\aK\alpha}
&\aP\aK\aP\aX\ar[d]_{\alpha}
&\aP\aW\aX\ar[l]_{\aP\upsilon}\ar[dl]^{P\omega}\\
&\aK\aC&\aP\aX
}
\]
commute up to natural transformation of maps of $\Gamma$-categories (on
the left) for all $\aC$ and up to op-lax natural transformation (on
the right) for all $\aX$. 

On the left, starting with an object $(y,x,g)$ in $\aW\aK\aC(\fpn)$,
the top left composite takes this to $\aK\alpha(y)$ and the diagonal arrow
takes this to $x$; the effect on maps in $\aW\aK\aC(\fpn)$ admits the
analogous description.  Since $\aK\alpha(\iota x)=x$, $\aK\alpha(g)$
is a map from $\aK\alpha(y)$ to $x$, which is natural in
$\aW\aK\aC(\fpn)$, and compatible with the $\Gamma$-structure.

On the right, consider an element $X=(X_{1},\dotsc,X_{s})$ in
$A\aW\aX(\aon)$, where $X_{i}=(y^{i},x^{i},g^{i})$ is an object in
$\aW\aX(\fpn_{i})$.  As per the remark following
Definition~\ref{defW}, we can write $y^{i}=(y^{i}_{I},f_{I,J})$ for
$y^{i}_{I}$ some object of $\aX(\fpm_{I})$ (thought of as
$A\aX(m)$ or $A\aX()$) for some $m_{I}$, where $I$ ranges over the
subsets of $\fun_{i}$.  The left down composite sends $X$ to 
\[
\alpha(y^{1},\dotsc,y^{s})=(y^{1}_{\fum_{\fun_{1}}},\dotsc,y^{s}_{\fum_{\fun_{s}}})
\]
since the symmetric monoidal product in $\aP\aX$ is concatenation.  An
analogous description applies to maps of $X$ in $\aP\aW\aX$.
The diagonal in the diagram sends $X$ to $(x^{1},\dotsc,x^{s})$ and we
have the map
\[
(g^{i}_{\fum_{\fun_{i}}})\colon (y^{i}_{\fum_{\fun_{i}}})\to 
(\iota x^{i}_{\fum_{\fun_{i}}})=(x^{i}).
\]
in $\aP\aX$.  This map is natural in $X$ in $\aP\aW\aX$ and is a strictly
monoidal natural transformation.
\end{proof}

\begin{proof}[Proof of Theorem~\ref{maincat}]
Given Propositions~\ref{propNsimp} and~\ref{propNsimpS},
Theorem~\ref{thmW}, and Corollaries~\ref{corPK} and~\ref{corKP}, it
suffices to show that the map $\upsilon \colon \aW\aX\to \aK\aP\aX$ is
always a stable equivalence.  Writing $[-,-]$ for maps in the homotopy
category obtained by formally inverting the weak equivalences, we need
to show that
\[
\upsilon^{*}\colon [\aK\aP\aX,\aZ]\to [\aW\aX,\aZ]
\]
is a bijection for every very special $\Gamma$-category $\aZ$.  Since
$\aK$ and $\aP$ preserve weak equivalences, they induce functors 
on the homotopy category.  Using this and the fact that $\upsilon$ is
a weak equivalence for a special $\Gamma$-category, we get a map
\[
R\colon [\aW\aX,\aZ]\to [\aK\aP\aX,\aZ]
\]
as follows: Given $f$ in $[\aW\aX,\aZ]$, the map
$Rf$ in $[\aK\aP\aX,\aZ]$ is the composite
\[
\xymatrix{%
\aK\aP\aX\ar[r]^{\aK\aP\omega^{-1}}
&\aK\aP\aW\aX\ar[r]^{\aK\aP f}
&\aK\aP\aZ\ar[r]^{\upsilon^{-1}}
&\aW\aZ\ar[r]^{\omega}&\aZ.
}
\]

To see that the composite map on $[\aW\aX,\aZ]$ is the identity,
consider the following diagram,
\[
\xymatrix{%
\aX&\aW\aX\ar[l]_{\omega}^{\sim}\ar[r]^{f}
&\aZ\\
\aW\aX\ar[u]^{\omega}_{\sim}\ar[d]_{\upsilon}
&\aW\aW\aX\ar[l]_{\aW\omega}^{\sim}\ar[u]^{\omega}_{\sim}
\ar[r]^{\aW f}\ar[d]_{\upsilon}
&\aW\aZ\ar[u]^{\omega}_{\sim}\ar[d]_{\upsilon}^{\sim}\\
\aK\aP\aX&\aK\aP\aW\aX\ar[l]^{\aK\aP\omega }_{\sim}
\ar[r]_{\aK\aP f}&\aK\aP\aZ
}
\]
which commutes by naturality.  We see that $\aW\omega$ is a weak
equivalence (as marked) by the two-out-of-three property since
$\omega$ is always a weak equivalence.
The map $R(f)\circ \upsilon$ is the composite map in the homotopy
category of the part of this diagram starting from the copy
of $\aW\aX$ in the first column and traversing maps and inverse maps
to $\aZ$ by going down, right 
twice, and then up twice; this agrees with the composite map in the
homotopy category obtained
by going up and then right twice, $f\circ \omega^{-1} \circ
\omega=f$. 

On the other hand, starting with $g$ in $[\aK\aP\aX,\aZ]$, then
\[
R(g\circ \upsilon)= 
\omega\circ \upsilon^{-1}\circ 
\aK\aP(g\circ \upsilon\circ \omega^{-1}).
\]
The solid arrow part of the diagram
\[
\xymatrix{%
\aK\aP\aX
&\aW\aK\aP\aX\ar[r]^{\aW g}\ar[d]^{\upsilon}_{\sim}\ar[l]_{\omega}^{\sim}
&\aW\aZ\ar[d]^{\upsilon}_{\sim}\ar[r]^{\omega}_{\sim}&\aZ\\
\aK\aP\aW\aX\ar[r]_{\aK\aP \upsilon}\ar[u]^{\aK\aP\omega}
&\aK\aP\aK\aP\aX\ar[r]_{\aK\aP g}\ar@{..>}[ul]^{\aK\alpha}
&\aK\aP\aZ
}
\]
commutes and Lemma~\ref{lemtechlem} implies that the whole diagram
commutes in the homotopy category.  By naturality of $\omega$, the
composite 
$\omega\circ \aW g\circ \omega^{-1}$ is $g$, and it follows that 
$R(g\circ \upsilon)$ is $g$.
\end{proof}

\section{Special $\Gamma$-spaces and non-completed $E_{\infty}$ spaces}
\label{secss}

This section explores the analogue in simplicial sets of the
construction of $\aP$ in small categories, which provides a functor
$E$ from $\Gamma$-spaces to $E_{\infty}$ spaces over the Barratt-Eccles
operad.  This section is entirely independent from the rest of the
paper and we have written it to be as self-contained as possible
without being overly repetitious.  We assume familiarity with
$\Gamma$-spaces, but not with $\Gamma$-categories or permutative
categories (except where we compare $E$ and $\aP$ in
Proposition~\ref{propcomp}). 

Definition~\ref{defA} describes a category $\aA$ whose objects are the
sequences of positive integers (including the empty sequence).
We think of a positive integer as a finite (unbased) set, and maps
between sequences $\aom=(m_{1},\dotsc,m_{r})$ and
$\aon=(n_{1},\dotsc,n_{s})$ are generated by permuting elements in
the sequence, maps of finite sets, and partitioning finite sets.  For
a $\Gamma$-space $X$, let $AX$ be the functor from $\aA$ to based
simplicial sets with 
\[
AX(\aon)=X(\fpn_{1})\times \dotsb \times X(\fpn_{s})
\]
for $s>0$ and $AX()=*$.  In terms of the maps in $\aA$, a permutation
of sequences induces the corresponding permutation of factors; a map
of finite unbased sets $\phi \colon \fun\to\fup$ induces the
corresponding map $X(\phi)$ (for the corresponding $\phi\colon
\fpn\to\fpp$); a partition $\fun=\fup_{1}\amalg \dotsb \amalg \fup_{t}$
induces the map
\[
X(\fpn)\to X(\fpp_{1})\times \dotsb \times X(\fpp_{t})
\]
induced by the maps $\fpn\to \fpp_{i}$ that pick out the elements of
the subset $\fpp_{i}$ and send all the other 
elements to the basepoint.  We consider the homotopy
colimit. 

\begin{defn}
Let $EX=\hocolim_{\aA} AX$.
\end{defn}

It is clear from the definition that $E$ preserves weak equivalences.
The proof of the remainder of the following theorem is identical to
the proof of Theorem~\ref{thmtech}.

\begin{thm}\label{thmtechspace}
$E$ preserves weak equivalences.  If $X$ is a special $\Gamma$-space,
then the inclusion of $X(\fpo)$ in $EX$ is a weak equivalence.
\end{thm}

Recall that the Barratt-Eccles operad $\aE$ has as its $n$-th
simplicial set $\aE(n)=NT\Sigma_{n}$ the nerve of the translation
category on the $n$-th symmetric group $\Sigma_{n}$, with operadic
multiplication induced by block sum of permutations.  For any
permutation $\sigma$ in $\Sigma_{n}$, we have a functor
\[
\sigma\colon \aA^{\times n}\to \aA
\]
induced by permutation and concatenation:
\[
\sigma(\aom^{1},\dotsc,\aom^{n})=
(m^{\sigma 1}_{1},\dotsc,m^{\sigma 1}_{r_{\sigma 1}},
m^{\sigma 2}_{1},\dotsc,m^{\sigma n}_{r_{\sigma n}}).
\]
Permutation induces a natural transformation
\[
AX^{\times n}\to AX
\]
covering $\sigma$; we therefore get an induced map on homotopy
colimits 
\[
\sigma_{*}\colon (EX)^{\times n}\to EX.
\]
For any other element $\sigma'\in \Sigma_{n}$, the permutation
$\sigma'\sigma^{-1}$ induces a natural transformation between 
functors 
\[
\sigma,\sigma'\colon \aA^{\times n}\to \aA,
\]
compatible with the natural transformations $AX^{\times n}\to AX$
covering them. These fit together to induce a map 
\begin{equation}\label{eqopact}
\aE(n)\times EX^{\times n}
\iso NT\Sigma_{n}\times \hocolim_{\aA^{\times n}}AX^{\times n}
\to EX.
\end{equation}
An easy check of the definitions proves the following proposition, a
restatement of Theorem~\ref{maineinfty}.

\begin{prop}
The maps~\eqref{eqopact} define an action of the operad $\aE$ on the
simplicial set $EX$.  This action is natural in maps of the $\Gamma$-space $X$.  Thus, $E$ defines a functor from $\Gamma$-spaces to
$E_{\infty}$ spaces over $\aE$.
\end{prop}

To compare the functor $E$ with the functor $\aP$, recall that the
nerve of a permutative category has the natural structure of an $\aE$
space with the map $\sigma_{*} \colon N\aC^{\times n}\to N\aC$ (for
$\sigma$ in $\Sigma_{n}$) induced by the permutation and the
permutative product.  In Section~\ref{secpf}, we reviewed the map from
the homotopy colimit of the nerve to the nerve of the Grothendieck
construction, which we can now interpret as a natural transformation
$EN\to N\aP$.  The following proposition is clear from explicit
description of the map in that section. 

\begin{prop}\label{propcomp}
For a $\Gamma$-category $\aX$, the canonical map $EN\aX\to N\aP\aX$ is a
map of $\aE$ spaces and a weak equivalence.
\end{prop}

Next we define the $\Gamma$-space version of the functor $EX$.  For this
we use the $\Gamma$-spaces $X_{n}$ defined by
\[
X_{n}(\fpm)=X(\fpn\fpm),
\]
where we use lexicographical ordering to make $\fpn\fpm$ a functor of
$\fpm$ from $\aF$ to $\aF$.  Taking advantage of the fact that
$\fpn\fpm$ is also a functor of $\fpn$, the construction $EX_{(-)}$
defines a functor from $\aF$ to simplicial sets.  However,
since we require $\Gamma$-spaces to satisfy $X(\fpz)=*$, we need a
reduced version.

\begin{defn}
Let $E^{\Gamma}X$ be the $\Gamma$-space with $E^{\Gamma}X(\fpn)$ the
based homotopy colimit in the category of based simplicial sets 
\[
E^{\Gamma}X(\fpn)=\hocolim^{*}_{\aA}AX_{n}.
\]
\end{defn}

The inclusions $\eta_{n}\colon X_{n}(\fpo)\to
E^{\Gamma}X(\fpn)$ now assemble to a map of $\Gamma$-spaces $X\to
E^{\Gamma}X$.
Since $\aA$ has an initial object $()$, the nerve $N\aA$ is
contractible.  The map
\[
EX_{n}=\hocolim_{\aA}AX_{n}\to (\hocolim_{\aA}AX_{n})/N\aA
=\hocolim^{*}_{\aA}AX_{n} = E^{\Gamma}X(\fpn)
\]
is therefore a weak equivalence.  Applying Theorem~\ref{thmtechspace}
objectwise to the map $X\to E^{\Gamma}X$, we get the following theorem.

\begin{thm}\label{thmtechgamma}
$E^{\Gamma}$ preserves weak equivalences.  If $X$ is a special $\Gamma$-space, then the natural map $X\to E^{\Gamma}X$ is a weak equivalence.
\end{thm}

We prove the following theorem at the end of the section.

\begin{thm}\label{thmspecial}
For a $\Gamma$-space $X$, $E^{\Gamma}X$ is a special $\Gamma$-space.
\end{thm}

Finally, we need one further variant of this construction.  Let
$AEX$ be the functor from $\aA$ to simplicial sets with
\[
AEX(\aon)=EX_{n_{1}}\times \dotsb \times EX_{n_{s}}
\]
and $AE()=EX_{0}=E*$: Although $EX_{(-)}$ is not a $\Gamma$-space, it is
an $\aF$-space (functor from $\aF$ to simplicial sets), and this is
all that is needed for the construction of the functor $AEX$.  Let
$E^{2}X$ be the simplicial set 
\[
E^{2}X=\hocolim_{\aA}AEX
\]
(homotopy colimit in the category of unbased simplicial sets).  The
map of $\aF$-spaces $EX_{(-)}\to E^{\Gamma}X(-)$ induces a weak
equivalence $AEX\to AE^{\Gamma}X$ and a weak equivalence $E^{2}X\to
E(E^{\Gamma}X)$.  

The advantage of $E^{2}X$ over $E(E^{\Gamma}X)$ is that we can
construct a map $E^{2}X\to EX$ as follows.  For each $\aom$ in $\aA$,
we have a map 
\begin{multline*}
AEX(\aom)=EX_{m_{1}}\times \dotsb \times EX_{m_{r}}
\iso 
\hocolim_{\aA^{\times r}}(AX_{m_{1}}\times \dotsb \times AX_{m_{r}})\\
\to
\hocolim_{\aA}AX=EX
\end{multline*}
induced by the functor $\rho_{\aom}\colon \aA^{\times r}\to \aA$,
defined by
\[
\rho_{\aom}\colon 
(\aon_{1},\dotsc,\aon_{r})\mapsto 
(m_{1}n_{1,1},\dotsc,m_{1}n_{1,s_{1}},m_{2}n_{2,1},\dotsc,m_{r}n_{r,s_{r}})
\]
(where $\aon_{i}=(n_{i,1},\dotsc,n_{i,s_{i}})$), together with the
canonical isomorphism
\[
AX_{m_{1}}(\aon_{1})\times \dotsb \times AX_{m_{r}}(\aon_{r})
\iso AX(\rho_{\aom}(\aon_{1},\dotsc,\aon_{r}))
\]
covering $\rho_{\aom}$.  These maps are compatible with maps $\aom$ in
$\aA$, and so induce a map $\alpha \colon E^{2}X\to EX$.  The
technical fact about this map we need is the following lemma, which is
an easy check of the construction.

\begin{lem}\label{lemtechspace}
The diagram
\[
\xymatrix{%
EX\ar[r]^{\eta}\ar[d]_{E\eta}\ar[dr]^{\id}&E^{2}X\ar[d]^{\alpha}\\
E^{2}X\ar[r]_{\alpha}&EX
}
\]
commutes where $\eta$ is induced by the inclusion of $EX$ as $AEX(1)$
and $E\eta$ is induced by the inclusion of $AX$ in $AEX$.
\end{lem}

Applying the lemma to $X_{n}$, we get a commuting diagram of
$\aF$-spaces.  We can turn this into a diagram of $\Gamma$-spaces by
taking the quotient by $EX_{0}=E*$ or $E^{2}X_{0}=E^{2}*$ at each
spot.  We then get a commutative diagram of $\Gamma$-spaces
\[
\xymatrix{%
E^{\Gamma}X\ar[r]^{\eta}\ar[d]_{E\eta}\ar[dr]^{\id}
&E^{2}X_{(-)}/E^{2}*\ar[d]^{\alpha}\\
E^{2}X_{(-)}/E^{2}*\ar[r]_{\alpha}&E^{\Gamma}X.
}
\]
We note that $E^{2}*$ is contractible, and since $E^{\Gamma}X$ is
special,  $\eta$ is weak equivalence.  It follows that all maps in the
diagram are weak equivalences.  Since both 
\[
\eta \colon E^{\Gamma}X\to E^{\Gamma}E^{\Gamma}X
\qquad \text{and}\qquad
E^{\Gamma}\eta \colon E^{\Gamma}X\to E^{\Gamma}E^{\Gamma}X
\]
factor through the corresponding map $E^{\Gamma}X\to E^{2}X/E^{2}*$, we
get the following proposition.

\begin{prop}
The maps $\eta$ and $E^{\Gamma}\eta$ from $E^{\Gamma}X$ to
$E^{\Gamma}E^{\Gamma}X$ coincide in the 
strict homotopy category of $\Gamma$-spaces, i.e., the homotopy category
obtained by formally inverting the objectwise weak equivalences.
\end{prop}

We use this observation to prove the following theorem, which together with
Theorems~\ref{thmtechgamma} and~\ref{thmspecial} imply
Theorem~\ref{maingamma}. 

\begin{thm}
For any $\Gamma$-space $X$,
$\eta \colon X\to E^{\Gamma}X$ is a stable equivalence.  Moreover 
$\eta$ is the initial map from $X$ to a special $\Gamma$-space in the
strict homotopy category of $\Gamma$-spaces.
\end{thm}

\begin{proof}
We need to show that for
any special $\Gamma$-space $Z$, the map $\eta$ induces a bijection
$[E^{\Gamma}X,Z]\to [X,Z]$ where $[-,-]$ denotes maps in the strict homotopy
category of $\Gamma$-spaces.  Since $E^{\Gamma}$ preserves weak
equivalences, it induces a functor on the strict homotopy category.
Given a map $g$ in $[X,Z]$, $E^{\Gamma}g$ is a map in
$[E^{\Gamma}X,E^{\Gamma}Z]$, and since $\eta \colon Z\to EZ$ is a weak
equivalence, we can compose with the map $\eta^{-1}$ in the strict
homotopy category to get an element $Rg=\eta^{-1}\circ E^{\Gamma}g$ in
$[E^{\Gamma}X,Z]$. By
naturality of $\eta$, $R$ is a retraction. By
examination of the solid arrow commuting diagram 
\[
\xymatrix{%
X\ar[r]^{\eta}
&E^{\Gamma}X\ar[r]^{g}\ar[d]^{\eta}_{\sim}
&Z\ar[d]^{\eta}_{\sim}\\
E^{\Gamma}X\ar[r]_{E^{\Gamma}\eta}\ar@{..>}[ur]^{\id}
&E^{\Gamma}E^{\Gamma}X\ar[r]_{E^{\Gamma}g}&E^{\Gamma}Z
}
\]
and applying the previous proposition, we see that $R$ is a bijection.
\end{proof}

The previous theorem also provides the final piece for the proof of Theorem~\ref{mainfouruncomp}.

\begin{proof}[Proof of Theorem~\ref{mainfouruncomp}]
The equivalence of~\enref{enmfc} and~\enref{enmfd} is
Theorem~\ref{thmuncomp} proved in the last section.  The previous
theorem proves the equivalence of~\enref{enmfb} and~\enref{enmfc}, and
\cite[1.8]{MayThomason} (and the argument for \cite[1.1]{KrizMay})
prove the equivalence of~\enref{enmfa} and~\enref{enmfc}.
\end{proof}

We close with the proof of Theorem~\ref{thmspecial}.  We thank Irene
Sami for help putting together this argument.

\begin{proof}[Proof of Theorem~\ref{thmspecial}]
It suffices to show that for every $j>0$, the map
\begin{equation}\label{eqpfspec}
EX_{j+1}\to EX_{j}\times EX
\end{equation}
is a weak equivalence.  Using $EX_{j}$ in place of
$E^{\Gamma}X(\fps{j})$ has the advantage that we can write
$EX_{j}\times EX$ as a homotopy colimit:
\[
EX_{j}\times EX\iso \hocolim_{\aA\times \aA}(AX_{j}\times AX).
\]
For clarity in formulas that follow, we will use brackets $[m]$ rather
than bold $\fpm$ to denote finite based sets.

The map~\eqref{eqpfspec} is induced by the diagonal functor
$\aA\to\aA\times \aA$ and the natural transformation
\[
AX_{j+1}(\aom)\to AX_{j}(\aom)\times AX(\aom).
\]
We get a map
\begin{equation}\label{eqinvspec}
EX_{j}\times EX\to EX_{j+1}
\end{equation}
induced by the concatenation functor $\aA\times \aA\to \aA$ and the
natural transformation
\[
AX_{j}(\aom)\times AX(\aon)\to AX_{j+1}(\aom\boxbin \aon)
\]
(where $\boxbin$ denotes concatenation), sending 
\[
X([jm_{1}])\times \dotsb \times X([jm_{r}])\to 
X([(j+1)m_{1}])\times \dotsb \times X([(j+1)m_{r}]) 
\]
by the map induced by the inclusion of 
$[j]$ in $[j+1]$, and the map
\[
X([n_{1}])\times \dotsb \times X([n_{s}])\to 
X([(j+1)n_{1}])\times \dotsb \times X([(j+1)n_{s}]) 
\]
induced by including the non-basepoint element $1$ of $[1]$ as the
element $j+1$ of  $[j+1]$.  We show that~\eqref{eqpfspec}
and~\ref{eqinvspec} are inverse generalized simplicial homotopy
equivalences.  

First we show that the composite on $E_{j+1}$ is (generalized
simplicial) homotopic to the identity.
We denote the composite on $E_{j+1}$ as $(D,d)$. It is induced by the
functor $D\colon \aA\to \aA$ that 
sends $\aom$ to the concatenation $\aom\boxbin\aom$ and the natural
transformation 
\[
d\colon X_{j+1}(m_{i})=
X([(j+1)m_{i}])\to X([(j+1)m_{i}])\times X([(j+1)m_{i}])
=X_{j+1}(m_{i},m_{i})
\]
induced in the first factor by sending the element
$j+1$ of
$[j+1]$ to the basepoint and induced in the second factor by
sending the elements $1,\dotsc,j$ of $[j+1]$ to the
basepoint.  

We construct a new map $(H,h)$ from $E_{n+1}$ to
itself and simplicial homotopies from $(H,h)$ to $(D,d)$ and from
$(H,h)$ to the identity as follows.
Let $H$ be the functor $\aA\to \aA$ that
sends $(m_{1},\dotsc,m_{r})$ to $((j+1)m_{1},\dotsc,(j+1)m_{r})$ and
let 
\begin{multline*}
h\colon AX_{j+1}(\aom)=AX((j+1)m_{1},\dotsc,(j+1)m_{r})\\
\to AX((j+1)^{2}m_{1},\dotsc,(j+1)^{2}m_{r})
=AX_{j+1}(H\aom)
\end{multline*}
be the natural transformation induced by the diagonal map in $\aF$
from $[j+1]$ 
to $[(j+1)^{2}]$; then the functor $H$ and natural transformation $h$
induce a map $(H,h)$ from $EX_{j+1}$
to itself.  

We have a natural transformation $\phi$ from $H$ to $D$ formed by
concatenation and permutation from the maps
\[
((j+1)m_{i}) \to (m_{i},m_{i})
\]
in $\aA$ sending collapsing the first $j$ copies of $\fum_{i}$ to the
first $\fum_{i}$ by the codiagonal map and sending the last copy of
$\fum_{i}$ onto the second $\fum_{i}$.  The composite map
\[
\xymatrix{%
AX_{j+1}(\aom)\ar[r]^{h}&AX_{j+1}(H\aom)\ar[r]^{AX_{j+1}(\phi)}
&AX_{j+1}(D\aom)
}
\]
is $d$.  Thus, the natural transformation $\phi$ induces a homotopy
between the maps $(H,h)$ and $(D,d)$ on the homotopy colimit
$E_{j+1}$.

Likewise, we have a natural
transformation $\psi$ from $H$ to the identity induced by the maps
$((j+1)m_{i})\to (m_{i})$ in $\aA$ that collapse the $j+1$
copies of $m_{i}$ by the codiagonal.  Since the composite 
\[
\xymatrix{%
AX_{j+1}(\aom)\ar[r]^{h}
&AX_{j+1}(H\aom)\ar[r]^{AX_{j+1}(\psi)}
&AX_{j+1}(\aom)
}
\]
is the identity, it follows that $\psi$ induces a homotopy 
from $H$ to the identity on $EX_{j+1}$.  This constructs the
generalized simplicial homotopy equivalence between the composite map
and the identity map on $EX_{j+1}$.

The argument for the other composite is easier: The composite on
$EX_{j}\times EX$ is induced by the functor $D_{2}$ from $\aA\times
\aA$ to itself
\[
D_{2}(\aom,\aon)=(\aom\boxbin \aon,\aom\boxbin \aon)
\]
and the natural transformation 
\[
d_{2}\colon AX_{j}(\aom)\times AX(\aon)\to
AX_{j}(\aom\boxbin \aon)\times AX(\aom\boxbin\aon)
\]
induced on the first factor by the inclusion of $\aom$ in
$\aom\boxbin\aon$ and on the second factor by the inclusion of $\aon$
in $\aom\boxbin\aon$.  Since these maps
\[
(\aom,\aon)\to(\aom\boxbin\aon,\aom\boxbin\aon)
\]
assemble to a natural transformation in $\aA\times \aA$ from the
identity functor to $D_{2}$, they induce a
homotopy on $EX_{j}\times EX$ between the identity and the map induced
by $D_{2},d_{2}$.
\end{proof}


\bibliographystyle{amsplain}
\def\noopsort#1{}\def\MR#1{}
\providecommand{\bysame}{\leavevmode\hbox to3em{\hrulefill}\thinspace}
\providecommand{\MR}{\relax\ifhmode\unskip\space\fi MR }
\providecommand{\MRhref}[2]{%
  \href{http://www.ams.org/mathscinet-getitem?mr=#1}{#2}
}
\providecommand{\href}[2]{#2}

\end{document}